\documentclass[12pt,a4paper]{article}
\usepackage{amsmath,amsfonts,amssymb,amsthm,mathrsfs,graphics,graphicx,
hyperref,mathabx,setspace,bbm,cleveref,xcolor,url,mathtools,tikz}
\usetikzlibrary{decorations.pathreplacing,angles,quotes,3d}
\usepackage[margin=1in]{geometry}
\usepackage[backend=biber, style=numeric-comp, url=false, isbn=false, doi=false, giveninits=true,maxbibnames=99]{biblatex}
\newtheorem{proposition}{Proposition}[section]
\newtheorem{theorem}[proposition]{Theorem}
\newtheorem{corollary}[proposition]{Corollary}
\newtheorem{lemma}[proposition]{Lemma}
\newtheorem{definition}[proposition]{Definition}

\theoremstyle{definition}
\newtheorem{remark}[proposition]{Remark}

\renewcommand{\H}{\mathbb{H}}
\newcommand{\E}{\mathbb{E}}
\newcommand{\N}{\mathbb{N}}
\newcommand{\R}{\mathbb{R}}
\renewcommand{\P}{\mathbb{P}}
\newcommand{\dint}{{\rm d}}
\newcommand{\one}{\mathbbm{1}}

\DeclareMathOperator{\Isom}{Isom}

\DeclareMathOperator{\capop}{cap}
\DeclareMathOperator{\vol}{vol}

\title{Critical Poisson hyperplane percolation in hyperbolic space has no unbounded cells}
\author{Tillmann B{\"u}hler\thanks{Institute of Stochastics, Karlsruhe Institute of Technology   \href{mailto:tillmann.buehler@kit.ed}{\texttt{tillmann.buehler@kit.edu}}}, Anna Gusakova\thanks{Institute for Mathematical Stochastics, University of M\"unster \href{mailto:gusakova@uni-muenster.de}{\texttt{gusakova@uni-muenster.de}}}, and Konstantin Recke\thanks{Mathematical Institute, University of Oxford, UK  \href{mailto:konstantin.recke@maths.ox.ac.uk}{\texttt{konstantin.recke@maths.ox.ac.uk}}}}
\date{}

\begin{document}

\maketitle
\begin{abstract}
    We show that tessellations of hyperbolic space by isometry-invariant Poisson processes of $(d-1)$-dimensional hyperplanes do not have an unbounded cell at the critical intensity. This extends a result by Porret-Blanc for the hyperbolic plane ({\em C.~R.~Acad.~Sci.~Paris, Ser.~I} {\bf 344} (2007)) to dimensions $d\ge3$. We also show that for intensities strictly below the critical intensity, infinitely many unbounded cells exist, while for intensities larger than or equal to the critical intensity, no unbounded cell exists. This completely describes the basic phase transition of this continuum percolation model. Our proof uses a method from discrete percolation theory which we adapt to the continuum and combine with specific computations for Poisson hyperplane processes.

\end{abstract}

\section{Introduction}

We study tessellations of hyperbolic space $\mathbb H^d$, $d\ge2$, induced by isometry-invariant Poisson processes of hyperplanes viewed as a continuum percolation model. Let us start with a precise description.
Let $A(d,d-1)$ be the set of hyperplanes in $\mathbb H^d$ (that is, \((d-1)\)-dimensional totally geodesic subspaces of \(\H^d\)).
There exists a non-zero isometry-invariant regular Borel measure $\mu_{d-1}$ on $A(d,d-1)$ which is unique up to scaling; see Section \ref{sec:intensity measure} for details and our normalization.
Let $\bar\xi$ be a Poisson process on $A(d,d-1)\times[0,\infty)$ with intensity measure $\mu_{d-1}\times d t$.
The {\bf Poisson hyperplane process with intensity} $\gamma>0$ is the point process $\xi_\gamma$ on $A(d,d-1)$ which consists of all $H$ such that $(H,t)\in \bar\xi$ and $t\le\gamma$.
Note that $\xi_\gamma$ is a Poisson process on $A(d,d-1)$ with intensity measure $\gamma\cdot\mu_{d-1}$.
The union of all hyperplanes in \(\xi_\gamma\) is a random closed subset of \(\H^d\) which we denote
\begin{equation} \label{def-Z}
\mathcal Z_\gamma := \bigcup_{H\in\xi_\gamma} H.
\end{equation}
We are interested in the {\bf vacant set}
\begin{equation} \label{def-V}
\mathcal V_\gamma := \mathbb H^d \setminus \mathcal Z_\gamma,
\end{equation}
which defines an isometry-invariant continuum percolation model subsequently referred to as {\bf Poisson hyperplane percolation with intensity $\gamma>0$}. 

Note that every connected component of $\mathcal V_\gamma$ is the interior of a closed, convex polyhedron with a possibly infinite number of sides (see \cite[Section 6.3]{Ra19} for the definition and more details).
In particular, the collection of (closures of) connected components of $\mathcal V_\gamma$ forms a tessellation of hyperbolic space.
We will refer to connected components shortly as {\bf cells} and denote by $\mathcal C_\gamma(z)$ the $\mathcal V_\gamma$-cell containing $z\in\H^d$ which is a.s.\ well defined. For a fixed origin $o\in\mathbb H^d$, define the {\bf zero cell} to be $\mathcal C_\gamma:=\mathcal C_\gamma(o)$. We then define the {\bf critical intensity}
\begin{equation} \label{def-critical intensity}
\gamma_c := \gamma_c(\mathbb H^d) \coloneqq \sup \big\{ \gamma>0 \colon \mathcal C_\gamma \, \,\text{is unbounded with positive probability} \big\}.
\end{equation}
By monotonicity, $\mathcal C_\gamma$ is unbounded with positive probability for $\gamma<\gamma_c$ and bounded almost surely for $\gamma>\gamma_c$. The first interesting fact about this model is that it exhibits a non-trivial {\em phase transition} in the sense that
\begin{equation}\label{eq:value_critical}
0<\gamma_c<\infty \quad \text{for every dimension} \, \, d\geq2.
\end{equation}
In fact, it is known that 

\begin{equation*}
    \gamma_c = \gamma_c(\mathbb H^d) = \frac{(d-1)^2\sqrt{\pi}\Gamma(\frac{d-1}{2})}{\Gamma(\frac d 2)};
\end{equation*}
see \cite[Theorem 3.3]{GKT} for a proof.
We may also point out that the {\em expected} volume of the zero cell is finite if and only if \(\gamma > \gamma_c\); see \cite[Theorem 7.1]{BHT25}.

Clearly, the most important follow-up question is to determine the behavior {\em at the critical point}, i.e., to decide whether $\mathcal C_{\gamma_c}$ is unbounded with positive probability or not. Using covering properties of the circle by random arcs, Porret-Blanc \cite{P07} proved that in dimension $d=2$, $\mathcal C_{\gamma_c}$ is bounded almost surely. The main result of this paper addresses the question for arbitrary dimension.

\begin{theorem} \label{thm:main}
    For \(d \geq 2\), the zero cell of Poisson hyperplane percolation with critical intensity is almost surely bounded.
\end{theorem}

The proof of Theorem~\ref{thm:main} is given in Section~\ref{sec:mainproof}. We apply this result to obtain the following complete description of the percolation phase transition for Poisson hyperplane percolation.  

\begin{corollary}[Description of the phase transition] \label{cor:phase} For $d\ge2$, if $N_\gamma$ denotes the number of unbounded cells in $\mathcal V_\gamma$, we have that almost surely
$$
N_\gamma = \begin{cases} 0, & \text{if} \, \, \gamma \geq \gamma_c, \\ \infty, & \text{if} \, \, \gamma<\gamma_c. \\ \end{cases}
$$
\end{corollary}

In fact, an interesting property of the model in dimensions $d > 2$ is that it exhibits a total of $d-1$ phase transitions, at distinct critical intensities, where the qualitative behavior changes substantially, see Proposition \ref{prop:multiple_phase} for details.

We may also point out that for $d=1$ as well as the analogously defined process on~$\mathbb R^d$, $d\ge2$, the zero cell is bounded a.s.~for every value of the intensity. 

\medskip

{\bf\noindent Typical cell of Poisson hyperplane tessellation.} 
A further motivation for Theorem~\ref{thm:main} comes from stochastic geometry. Indeed, the vacant set $\mathcal{V}_{\gamma}$ is closely related to a classical model of stochastic geometry called Poisson hyperplane tessellation. More precisely, the Poisson hyperplane tessellation $\mathcal{H}_{\gamma}$ is the collection of closures of connected components of $\mathcal{V}_{\gamma}$. For the purposes of this section, we will refer to these also as cells. 

One of the key objects associated to $\mathcal{H}_{\gamma}$ is the typical cell, which may be intuitively seen as a random cell of the isometry invariant tessellation, chosen ``uniformly at random''.
For a rigorous definition of the law of the typical cell, we refer the reader to \cite[Theorem~2.1]{HLS24} and \cite[Section~5.4]{Herold2021}, or the more general treatment of the matter in \cite[Section 8]{Last10}.

We now give some intuition for the notion of the typical cell:
We can associate to each cell $\mathcal{C}\in \mathcal{H}_{\gamma}$ a center point $z(\mathcal{C}):=z(\mathcal{H}_{\gamma},\mathcal{C})\in \H^d$, where the function $z$ is measurable and isometry-covariant, i.e., $z(g\mathcal{H}_{\gamma},g\mathcal{C})=gz(\mathcal{H}_{\gamma},\mathcal{C})$ for any $g\in{\rm Isom}(\mathbb{H}^d)$, where ${\rm Isom}(\mathbb{H}^d)$ denotes the group of isometries of $\mathbb{H}^d$.
These center points form an isometry-invariant point process. For each cell $\mathcal{C}$ we consider its ``shifted''  version $\mathcal{C}_{z(\mathcal{C})}=g_{z(\mathcal{C})}\mathcal{C}$, where $g_{z(\mathcal{C})}$ is chosen uniformly from the set of isometries that map $z(\mathcal{C})$ to~$o$ (see \cite{Last10} for details).
Thus, we can associate to the tessellation $\mathcal{H}_{\gamma}$ a marked point process 
\[
\eta_z:=\sum_{\mathcal{C}\in \mathcal{H}_{\gamma}}\delta_{(z(\mathcal{C}),\mathcal{C}_{z(\mathcal{C})})}
\]
where the points $z(\mathcal{C})$ represent the positions of cells and the marks $\mathcal{C}_{z(\mathcal{C})}$ represent the shape of cells.
Finally, let $\eta^o_z$ be the Palm version of $\eta_z$, which may be seen as a version of $\eta_z$ conditioned on the fact that there exists a point $(o,\mathcal{C}_o)$ in the process.
The point $(o,\mathcal{C}_o) \in \eta_z^o$ is called the typical point of $\eta_z$ and the corresponding mark $\mathcal{C}_o$ the typical cell of Poisson hyperplane tessellation $\mathcal{H}_{\gamma}$. 

The above approach works if all cells of $\mathcal{H}_{\gamma}$ are bounded, since in this case we may choose a center $z(\mathcal{C})$ to be, for example, the center of the smallest ball containing $\mathcal{C}$.
For Poisson hyperplane tessellations with $\gamma > \gamma_c$, results in this direction have been obtained in \cite[Section 5.4]{Herold2021}.
Theorem~\ref{thm:main} thus implies that the typical cell (in the classical sense, described above) of Poisson hyperplane tessellation $\mathcal{H}_{\gamma}$ can be defined at $\gamma_c$.

On the other hand, there is no measurable covariant function $z$ associating a center point to unbounded cells (see Lemma~\ref{lm:ChoosingPoints}).

\begin{remark}[Volume of the typical cell at criticality]
    While the expected volume of the zero cell of $\mathcal{H}_{\gamma}$ is infinite for $\gamma=\gamma_c$ (see~\cite[Theorem~7.1]{BHT25}), the expected volume of the typical cell is finite.
    In fact, the mean volume of the typical cell is always given by $1/D$, where $D$ denotes the cell-intensity, i.e., the mean number of cell center points per unit volume;
    see \cite[Theorem 5.4.3]{Herold2021} (with $f\equiv 1$) for this statement in hyperbolic space and \cite[Equation (8.11)]{Last10} for homogeneous spaces.
\end{remark}
\begin{remark}[Situation in $\H^2$]
    In dimension $2$, the cell intensity (also called \emph{face-intensity} in the planar case) $D = D_\text{F}$ can be explicitly calculated for $\gamma \geq \gamma_c(\H^2)$.
    From now on, let $\gamma \geq \gamma_c(\H^2)$.
    Then by \cite[Theorem 5.5]{BS01}, the following ``Euler formula'' holds:
    \begin{equation*}
        2\pi (D_\text{F} - D_\text{E} + D_\text{V}) = -1,
    \end{equation*}
    where $D_\text{V}$ is the mean number of vertices (in our context, points where two lines intersect) per unit volume and $D_\text{E}$ is $1/2$ times the mean sum of degrees of vertices per unit volume (in our case $D_\text{E} = 2 D_\text{V}$, since all vertices have degree four).
    The face density is defined in \cite{BS01} in a way different from (but equivalent to) our definition above.
    The Euler formula now yields $D_\text{F} = D_\text{V} - \frac{1}{2\pi}$.
    Finally, the quantity $D_\text{V}$ can easily be calculated for arbitrary $\gamma > 0$ using the Crofton formula (see \cite[Theorem 1]{HHT21} or \cite[Equation (1.4)]{BHT23}):
    \[ D_\text{V} = \frac{\omega_1\omega_3}{2\omega_2^2} \gamma^2 = \frac{1}{\pi} \gamma^2, \]
    where $\omega_j \coloneqq 2\pi^{j/2}/\Gamma(j/2) $ denotes the surface area of the Euclidean unit ball in $\R^j$.
    It follows for $\gamma \geq \gamma_c(\H^2) = \pi$ that $D_\text{F} = (2\gamma^2 - 1)/(2\pi)$.
    Consequently, the expected area of the typical cell is $(2\pi)/(2\gamma^2-1)$.
    At criticality, setting $\gamma = \gamma_c$ yields $(2\pi)/(2\pi^2-1) \approx 0.335$.
\end{remark}

\medskip

{\bf\noindent Historical remarks.} Theorem~\ref{thm:main} is new in dimensions $d\ge3$. For the hyperbolic plane (i.e., $d=2$), it was first proved in~\cite{P07}. For completeness, we provide a brief outline of that proof and the reason it requires the dimension to be~$d=2$ following the exposition in~\cite[Section 4]{GKT}. 
Roughly speaking, since the boundary of $\mathbb H^2$ can be identified with the circle $\mathbb{S}^1$ in the Poincaré disc model, the question can be reduced to the problem of determining whether a certain sequence of random arcs covers the circle almost surely. The latter problem admits a remarkable solution due to Shepp~\cite{S72}: Suppose that $\ell_1,\ell_2,\ldots \in(0,1)$ is a non-increasing deterministic sequence and that $L_1,L_2,\ldots$ are random arcs of length $2\pi\ell_1,2\pi\ell_2,\ldots$ placed uniformly and independently at random on the circle. Then the circle is covered infinitely often almost surely if and only if
$$
\sum_{n=1}^\infty \frac{1}{n^2}e^{\ell_1+\ldots+\ell_n} = \infty,
$$
and otherwise there exist uncovered points with positive probability.
We refer the reader to~\cite{GKT} for more information about the history of this question as well as a proof that Shepp's condition holds in the planar critical case. To generalize this approach to $d\ge3$, we have to understand the more general problem of covering metric spaces by random sets, specifically whether the $(d-1)$-dimensional unit sphere is covered by a sequence of randomly placed spherical caps. This problem has been studied in several works including~\cite{H73,FJJS18}; we refer to the latter paper for an overview of the literature. 
Specifically for the question considered in this paper, a special case of a result of Hoffmann-J{\o}rgensen~\cite{H73} can be applied to determine whether the critical zero cell is bounded everywhere {\em except} at the critical intensity; see~\cite[Theorem~4.1 \& Remark~4.3]{GKT} for details.

Finally, for $d=2$, alternative proofs of Theorem~\ref{thm:main} have been given by Benjamini, Jonasson, Schramm and Tykesson~\cite[Proposition 6.1]{BJST09} and Tykesson and Calka~\cite{TC13} using continuum percolation theoretic arguments different from the approach in this paper. We explain and implement our approach in Section~\ref{sec:mainproof}. See~\cite{TC13} also for further results about this model in the hyperbolic plane.

\medskip

{\bf\noindent Organization.} The rest of the paper is organized as follows. Section~\ref{sec:preliminaries} is a preliminary section in which we collect some facts regarding hyperbolic space, the invariant measure on the space of hyperplanes in $\mathbb{H}^d$ and the Mass Transport Principle. Section~\ref{sec:mainproof} is devoted to the proof of our main results Theorem \ref{thm:main} and Corollary \ref{cor:phase}. In Section~\ref{sec:MultiplePhases} we study geometric properties of unbounded cells and establish the aforementioned existence of further critical intensities at which the system undergoes a phase transition.

\medskip

{\bf\noindent Acknowledgement.} We thank Chiranjib Mukherjee vor valuable discussions. For the purpose of open access, the authors have applied a CC BY public copyright licence to any author accepted manuscript arising from this submission. AG and TB were supported by the DFG priority program SPP 2265 \textit{Random Geometric Systems}.
AG was also supported by Germany's Excellence Strategy  EXC 2044 -- 390685587, \textit{Mathematics M\"unster: Dynamics - Geometry - Structure}.

\section{Preliminaries} \label{sec:preliminaries}

In this section, we provide background on hyperbolic space, including in particular the isometry-invariant measure on the space of hyperplanes and the Mass Transport Principle.

\subsection{Hyperbolic space}

Throughout this paper, $(\H^d,g_{\H^d})$ denotes the unique $d$-dimensional, simply connected Riemannian manifold (with metric $g_{\H^d}$) of constant sectional curvature $-1$.
We write $d_{\H^d}$ and $\vol_{\H^d}$ (or simply $\vol$) for the distance function and volume measure induced by the metric.
For $1\leq k\leq d-1$, a $k$-plane (or $k$-dimensional plane) is a $k$-dimensional totally geodesic subspace of $\H^d$.
In particular, a hyperplane is a $(d-1)$-dimensional plane.
An open ball in $\mathbb{H}^d$ with center $x$ and radius $r>0$ will be denoted by $B(x,r)$.

There exist a variety of models of hyperbolic space, which facilitate computations by identifying $\H^d$ with a subset of Euclidean space.
In this article we will use the Klein model (also called projective disc model or Beltrami--Klein model), in which $\H^d$ is identified with the Euclidean unit ball
\[
B_{\mathbb{R}^d}(o,1):=\{x\in\mathbb{R}^d\colon \|x\|<1\},
\]
where we write $o=(0,\ldots,0)$ for the origin.
The ideal boundary $\partial \mathbb{H}^d$ is identified with the unit sphere $\mathbb{S}^{d-1}=\partial B_{\mathbb{R}^d}(o,1)$.
The metric $g^\text{Kl}$ in the Klein model is given by
\[ g^\text{Kl}_x(u,v) = \frac{(1-|x|^2)\langle u,v \rangle + \langle x,u \rangle \langle x,v\rangle}{(1-|x|^2)^2},\quad x \in B_{\R^d}(o,1),\quad u,v \in T_xB_{\R^d}(o,1),  \]
where the tangent space $T_x B_{\R^d}(o,1)$ is identified with $\R^d$ and we write $\langle \, \cdot\, , \, \cdot \, \rangle$ for the Euclidean inner product (cf.~\cite[Theorem 6.1.5]{Ra19}).
In particular it follows for the distance function $d_{\H^d}^\text{Kl}$ that
\begin{equation}\label{eq:DistanceKlein}
d_{\mathbb{H}^d}^{\rm Kl}(o,v)={\rm artanh}(\|v\|), \quad v \in B_{\R^d}(o,1).
\end{equation}
A $k$-plane in the Klein model is the nonempty intersection of a Euclidean $k$-plane with $B_{\R^d}(o,1)$ (cf.\cite[Theorem 6.1.4]{Ra19}).

For more details on this model, we refer the reader to \cite[Section 6.1]{Ra19}.


\subsection{Intensity measure for hyperplanes and normalization} \label{sec:intensity measure} 



Recall that $A(d,d-1)$ denotes the space of hyperplanes in $\mathbb{H}^d$ and the intensity measure of the hyperplane process \(\xi_\gamma\) is given by \(\gamma \cdot \mu_{d-1}\), where the measure \(\mu_{d-1}\) on \(A(d,d-1)\) is defined as
\begin{equation}\label{eq:def_intensity_measure}
    \mu_{d-1}(\,\cdot\,) = \int_{G(d,1)} \int_L \cosh^{d-1}(d_{\mathbb{H}^d}(x,o)) \one\{H(L,x) \in \,\cdot\,\} \,\mathcal{H}^1(\dint x) \,\nu_1(\dint L).
\end{equation}
This is the isometry-invariant measure on the space of hyperplanes in $\H^d$, normalized as in \cite{HHT21,BHT23,BHT25}.
Here, \(G(d,1)\) is the space of lines (that is, $1$-planes of $\H^d$) containing some arbitrary fixed origin \(o\), \(\nu_1\) denotes the unique Borel probability measure on \(G(d,1)\) invariant under isometries fixing \(o\), \(\mathcal{H}^1\) denotes the \(1\)-dimensional Hausdorff measure and \(H(L,x)\) is the hyperplane orthogonal to \(L\) at \(x\).
We refer to \cite{HHT21,BHT23,BHT25} for more details.

In the Klein model, it follows from \eqref{eq:DistanceKlein} that the corresponding measure $\mu_{d-1}^{\rm Kl}$ is given by
\begin{equation}\label{eq:KleinHyperplaneModel}
\mu_{d-1}^{\rm Kl}(\,\cdot\,)=2\int_{\mathbb{S}^{d-1}}\int_{0}^1(1-t^2)^{-{d+1\over 2}}\one\{H'(u,t) \cap B_{\R^d}(o,1) \in \,\cdot\,\} \,\dint t \,\sigma_{d-1}(\dint u),
\end{equation}
where $H'(u,t)$ is a Euclidean hyperplane with normal vector $u\in \mathbb{S}^{d-1}$ and at distance $t>0$ in direction $u$ from the origin and $\sigma_{d-1}$ is the spherical Lebesgue measure, normalized so that $\sigma_{d-1}(\mathbb S^{d-1}) = 1$.

\subsection{Hyperbolic Mass Transport Principle}\label{sec:MassTransport}

A (non-negative) Borel measure $\mu$ on $\mathbb H^d \times\mathbb H^d$ is called {\bf diagonally-invariant} if
$$
\mu(g A\times g B)=\mu(A\times B)
$$
for all measurable subsets $A,B$ of $\mathbb H^d$ and $g\in{\rm Isom}(\mathbb H^d)$.

\begin{theorem}[Mass Transport Principle, {cf.~\cite{BS01}}] \label{theorem:MTP} Let $\mu$ be a diagonally-invariant Borel measure on $\mathbb H^d\times\mathbb H^d$. If $\mu(A\times\mathbb H^d)<\infty$ for some open $A\subset\mathbb H^d$, then 
$$
\mu(B\times\mathbb H^d)=\mu(\mathbb H^d\times B)
$$
for every measurable $B\subset\mathbb H^d$. Moreover, there exists $c\ge0$ such that $\mu(B\times\mathbb H^d)=c \cdot \vol(B)$ for every measurable $B\subset\mathbb H^d$.
\end{theorem}

\begin{proof}
    This is stated and proved for the hyperbolic plane in \cite[Theorem 5.2]{BS01} and this proof extends to $\mathbb H^d$, $d\ge3$. See also~\cite[Theorem 8.47]{LP16}.
\end{proof}

Here is a typical application of the Mass Transport Principle (cf.~\cite[Example 8.6]{LP16}).

\begin{lemma} \label{lm:ChoosingPoints} Let $\mathcal S$ be an isometry-invariant random open subset of $\mathbb H^d$. Suppose that, with positive probability, $\mathcal S$ has some connected component with infinite volume. Then there does not exist an isometry-invariant coupling $(\mathcal S,\mathcal Y)$ where $\mathcal Y$ is a point process in~$\mathbb H^d$ with the property that for every connected component $\mathcal C$ of $\mathcal S$, $|\mathcal Y\cap \mathcal C|=1$.
\end{lemma}
\begin{proof} Suppose that a coupling $(\mathcal S,\mathcal Y)$ as above is given. For each connected component  $\mathcal C$ of $\mathcal S$,  let $y(\mathcal C)$ be the unique point of $\mathcal Y$ in $\mathcal C$. Let $\mu$ be the measure on $\mathbb H^d\times\mathbb H^d$ uniquely determined by 
$$
\mu(A\times B):= \mathbb E\bigg[ \sum_{\mathcal C} {\rm vol}(\mathcal C \cap A) \mathbf 1_{y(\mathcal C)\in B} \bigg] \qquad (A,B\subset \mathbb H^d \text{ measurable}),
$$
where the sum runs over all connected components. Then $\mu$ is diagonally-invariant with 
$$
\mu(B(o,r)\times\mathbb H^d) \le {\rm vol}(B(o,r))
$$
for every $r>0$. Theorem~\ref{theorem:MTP} yields that 
$$
\mu(\mathbb H^d\times B(o,r)) = \mathbb E\bigg[ \sum_{\mathcal C} {\rm vol}(\mathcal C) \mathbf 1_{y(\mathcal C)\in B(o,r)} \bigg]<\infty
$$
for every $r>0$. But choosing $r$ sufficiently large, there is a positive probability that $y(\mathcal C)\in B(o,r)$ for some $\mathcal C$ with infinity volume. This is a contradiction.
\end{proof}

Lemma~\ref{lm:ChoosingPoints} shows that the standard approach for defining the typical cell of a Poisson hyperplane tessellation with bounded cells does not work if there are cells of infinite volume. Indeed, this follows immediately from Lemma~\ref{lm:ChoosingPoints} if the center of each cell lies in the cell and otherwise we may project the center to a point in the cell in an isometry-invariant way. Note that it also follows that there does not exist a {\em generalized center function}, i.e., an isometry-covariant function assigning to each non-trivial (i.e., non-empty and non-full) closed convex subsets of $\mathbb H^d$ one of its points.

\section{Encounter points in hyperplane tessellations} \label{sec:mainproof}

In this section, we prove Theorem~\ref{thm:main} following roughly the approach in~\cite{BLPS99}.
This approach relies on the following fundamental concept, which goes back to the influential work of Burton and Keane~\cite{BK89} and is a well-known tool in discrete percolation theory~\cite{LP16}. 

\begin{definition}[Encounter point]
Let $\mathcal S$ be a random open subset of $\mathbb H^d$ and let $\mathcal Y$ be an independent simple point process on $\mathbb H^d$. Let $r>0$. A point $y\in \mathcal Y$ is called an {\bf $r$-encounter point} if the following conditions hold.
\begin{itemize}
\item The set $\mathcal C_{\mathcal S}(y)\setminus B(y,r)$, where $\mathcal C_{\mathcal S}(y)$ for $y\in\mathcal S$ denotes the connected component of $y$ in $\mathcal S$ and $\mathcal C_{\mathcal S}(y):=\varnothing$ if $y\notin\mathcal S$, has at least three unbounded connected components.
\item There are no other points of $\mathcal Y$ in $B(y,2r)$.
\end{itemize}
\end{definition}

The main geometrical ingredient in the proof of Theorem~\ref{thm:main} is the following.

\begin{lemma}[Unbounded hyperplane tessellations admit encounter points] \label{lem:encounter_pos_prob}
    Let $d\ge2$ and let $\mathcal V_\gamma$ denote Poisson hyperplane percolation in $\mathbb H^d$ with intensity $\gamma>0$. Let $\mathcal Y$ be an independent Poisson process on $\mathbb H^d$ with intensity measure ${\rm vol}(\cdot)$. Suppose that there exists an unbounded cell in \(\mathcal{V}_\gamma\) with positive probability. 
    Then, for some $r>0$, there exists an $r$-encounter point with positive probability.
\end{lemma}

We emphasize that the standard proof of existence of encounter points for Bernoulli percolation in the non-uniqueness phase \cite{BK89} does not apply in our setting. This is because the vacant set $\mathcal{V}_\gamma$ does not satisfy versions of insertion and deletion tolerance where we modify the configuration only inside a bounded window (as for instance in Boolean percolation). This difficulty arises despite the fact that the underlying Poisson process of hyperplanes is insertion and deletion tolerant. To underscore this point, let us sketch an attempt to prove the existence of encounter points inspired by the Bernoulli percolation argument: Suppose that unbounded cells exist. Then we may choose a large radius $r>0$ such that, with positive probability, $B(o,r)$ intersects at least three {\em distinct} unbounded cells $\mathcal C_1,\mathcal C_2$ and $\mathcal C_3$. Now, we would like to modify the configuration inside $B(o,r)$ to connect these cells via the origin and thus turning the origin into an encounter point. However, since hyperplanes are unbounded, the removal or addition of any hyperplane also changes the configuration arbitrarily far {\em outside} of $B(o,r)$. Hence even though the hyperplane process is deletion tolerant, which in particular implies that removing all hyperplanes intersecting $B(o,r)$ is an absolutely continuous operation, this construction does not guarantee that the new cells containing $\mathcal C_1,\mathcal C_2$ and $\mathcal C_3$ are disjoint after removing $B(o,r)$. Our proof of Lemma~\ref{lem:encounter_pos_prob} instead relies on specific considerations for Poisson hyperplane processes and is given in Section~\ref{sec:ExistenceEncounter} below.

In the remainder of this section, we give a proof of Theorem~\ref{thm:main} assuming Lemma~\ref{lem:encounter_pos_prob}. This proof uses fundamentals about encounter points which were developed for the continuum setting in the recent work~\cite{MS25} (note that encounter points are called {\em trifurcations} there). We now recall these fundamentals.

\begin{lemma}\label{lm:MoreEncounter}
Let $\mathcal S$ be an isometry-invariant random open subset of $\mathbb H^d$ and let $\mathcal Y$ be an independent isometry-invariant simple point process on $\mathbb H^d$. Assume that the expected number of connected components of $\mathcal S \cap B(o,1)$ is finite. Let $r>0$. Then the following holds almost surely: If $y\in \mathcal Y$ is an $r$-encounter point, then every unbounded connected component of $\mathcal C_{\mathcal S}(y)\setminus B(y,r)$ contains infinitely many $r$-encounter points.
\end{lemma}
\begin{proof} This follows from the Mass Transport Principle (Theorem~\ref{theorem:MTP}); see \cite[Lemma~4.2]{MS25} for details.
\end{proof}

Lemma~\ref{lm:MoreEncounter} allows us to define an isometry-invariant forest on encounter points, which will be instrumental in our proof of Theorem~\ref{thm:main}. More precisely, in the setting of Lemma~\ref{lm:MoreEncounter}, consider independent ${\rm Unif}[0,1]$-distributed labels for the points of $\mathcal Y$ and define the random graph $\mathcal F:=(V,E)$ with vertex set $V$ consisting of all $r$-encounter points and edge set $E$ obtained as follows. First, for each $y\in \mathcal Y$, draw an oriented edge $(y,y')$ from $y$ to $y'\in \mathcal Y$ if $y'$ is the closest $r$-encounter point to $y$ (with respect to the distance in $\mathcal S$) among the $r$-encounter points in the same connected component of $\mathcal C_{\mathcal S}(y)\setminus B(y,r)$ as $y'$ -- in case of multiple choices, choose $y'$ with minimal label. Let $E$ consist of all edges obtained from this set of oriented edges by forgetting the orientation. We now summarize the main properties of this construction.

\begin{proposition}[Forest on encounter points]\label{prop:ForestEncounter}
Let $\mathcal S$ be an isometry-invariant random open subset of $\mathbb H^d$ and let $\mathcal Y$ be an independent isometry-invariant simple point process on~$\mathbb H^d$. Assume that the expected number of connected components of $\mathcal S \cap B(o,1)$ is finite. Let $r>0$. Let $\mathcal F=(V,E)$ be the graph on $r$-encounter points constructed above. Then the following hold almost surely.
\begin{itemize}
\item[{\rm(i)}] Every $v\in V$ has at least three and only finitely many neighbors.
\item[{\rm(ii)}] If $\mathcal F$ is non-empty, then it is a forest consisting of only infinite trees.
\end{itemize}
\end{proposition}

\begin{proof} This can be shown as in~\cite[p.~1352]{BLPS99}; see~\cite[Lemma 4.3 \& Lemma 4.4]{MS25} for details.
\end{proof}

With these preparations established, we proceed to the proof of Theorem~\ref{thm:main} assuming Lemma~\ref{lem:encounter_pos_prob}.

\begin{proof}[Proof of Theorem~\ref{thm:main}] Suppose that the theorem fails, i.e., that critical Poisson hyperplane percolation has unbounded cells with positive probability. Let $\mathcal Y$ be an independent Poisson process on $\mathbb H^d$ with intensity measure ${\rm vol}(\cdot)$. By Lemma~\ref{lem:encounter_pos_prob}, we can choose $r>0$ such that, with positive probability, there exists an $r$-encounter point. By ergodicity (Corollary~\ref{cor:ergodic}), there exist $r$-encounter points almost surely. We fix this $r$ and will refer to $r$-encounter points shortly as encounter points throughout this proof. 

To see that the expected number of connected components of $\mathcal{V}_{\gamma_c} \cap B(o,1)$ is finite,
observe that the number of such connected components can be bounded from above by~$2^N$, where $N$ is the number of hyperplanes intersecting the ball $B(o,1)$. Since the latter is Poisson distributed with finite mean, $2^N$ has finite expectation.
Hence we can apply Proposition~\ref{prop:ForestEncounter}.

Let $\mathcal F:=(V,E)$ denote the isometry-invariant random forest on encounter points from Proposition~\ref{prop:ForestEncounter}. Recall that $V\subset \mathcal V_{\gamma_c}$. We will subsequently use this embedding into~$\mathbb H^d$ without further mention.

Let $\varepsilon>0$. Let $\mathcal Z_\varepsilon$ be the union of hyperplanes from an independent Poisson hyperplane process with intensity $\varepsilon$. Then 
$$
\mathcal V_{\gamma_c} \setminus \mathcal Z_\varepsilon \overset{(d)}{=} \mathcal V_{\gamma_c+\varepsilon}.
$$
By definition of $\gamma_c$, $\mathcal V_{\gamma_c} \setminus \mathcal Z_\varepsilon$ has no unbounded cells almost surely.

 We define the configuration of a bond percolation model $\omega_\varepsilon$ on $\mathcal F$ as follows. For $[u,v]\in E$, let $[u,v]\in\omega_\varepsilon$ if and only if $u$ and $v$ are in the same cell of $\mathcal V_{\gamma_c} \setminus \mathcal Z_\varepsilon$. 

Since $V$ is locally finite and $\mathcal V_{\gamma_c} \setminus \mathcal Z_\varepsilon$ has no unbounded cells almost surely, $\omega_\varepsilon$ has no infinite components almost surely. 

For $v\in V$, let $K_{\varepsilon}(v)$ denote the $\omega_{\varepsilon}$-cluster of $v$ and let $\partial K_{\varepsilon}(v)$ denote the inner vertex boundary, i.e., the set of $u\in K_{\varepsilon}(v)$ such that there exists $w\notin K_{\varepsilon}(v)$ with $[u,w]\in E$. Since $\omega_{\varepsilon}$ has only finite components almost surely,
$$
m(v,u):=\mathbf 1\{v\in V, u\in \partial K_{\varepsilon}(v)\} \frac{1}{|\partial K_{\varepsilon}(v)|}
$$
is well-defined. Note that for every $v\in V$, $\sum_{u\in V} m(v,u)=1$. 

For measurable $A,B\subset\mathbb H^d$, let 
$$
\mu(A\times B):= \mathbb E \bigg[ \sum_{v,u\in V} \mathbf 1\{v\in A, u\in B\} m(v,u) \bigg]
$$
be the expected amount of mass transported from vertices in $A$ to vertices in $B$.
Then $\mu$ can be extended  to a non-negative diagonally-invariant Borel measure on $\mathbb H^d\times \mathbb H^d$. Moreover, for $A:=B(o,r)$, we have that $\mu(A\times\mathbb H^d)\le 1$ because there can be at most one encounter point in $A$. In fact,
\begin{equation} \label{equ:ExpMassOut2}
\mu(B(o,r)\times\mathbb H^d) = \mathbb P( V\cap B(o,r)\ne \varnothing),
\end{equation}
i.e.~the expected mass transported out of the $r$-ball equals the probability of seeing an encounter point in this region.

The Mass Transport Principle (Theorem~\ref{theorem:MTP}) yields that
\begin{equation} \label{equ:MTPcontr2}
\mu(B\times\mathbb H^d)=\mu(\mathbb H^d\times B)
\end{equation}
for every measurable $B\subset\mathbb H^d$ and in particular $B=B(o,r)$. To compute $\mu(\mathbb H^d\times B(o,r))$, we start with the following observation. For every $v\in V$,
$$
\sum_{u\in V} m(u,v) = \mathbf 1\{ v\in \partial K_\varepsilon(v) \}\frac{|K_\varepsilon(v)|}{|\partial K_\varepsilon(v)|} \le 2 \cdot \mathbf 1\{v \in \partial K_\varepsilon (v)\}
$$
almost surely because $\mathcal F$ is a forest such that every vertex has degree at least three (in fact, in a tree such that every vertex has degree at least three, every non-empty finite subset $K$ of vertices satisfies $2|\partial K|\ge |K|$).
Since $B(o,r)$ can contain at most one encounter point, we obtain that
\begin{equation} \label{equ:ExpMassIn}
\mu(\mathbb H^d\times B(o,r)) \le 2 \mathbb P\big( |V\cap B(o,r)| = 1, v\in V\cap B(o,r) \ \mbox{satisfies} \ v\in \partial K_\varepsilon(v)\big) 
\end{equation}
Conditionally on $\mathcal F$ and $v\in V\cap B(o,r)$, the probability that $v\in \partial K_\varepsilon(v)$ tends to zero as $\varepsilon\to0$ because $v$ is connected by bounded length paths to finitely many neighbors, which are unlikely to intersect $\mathcal Z_\varepsilon$ for small $\varepsilon$. The bounded convergence theorem yields
$$
\lim_{\varepsilon\to0} \mathbb P\big( |V\cap B(o,r)| = 1, v\in V\cap B(o,r) \ \mbox{satisfies} \ v\in \partial K_\varepsilon(v)\big) =0,
$$
which together with \eqref{equ:ExpMassOut2} contradicts \eqref{equ:MTPcontr2} because $\mathbb P( V\cap B(o,r)\ne \varnothing)>0$, independently of $\varepsilon$.
\end{proof}

\subsection{Proof of Lemma \ref{lem:encounter_pos_prob}} \label{sec:ExistenceEncounter}

Fix a \(\gamma > 0\) such that with positive probability \(\mathcal{V}_\gamma\) contains an unbounded cell. Equivalently, we may assume that the zero-cell \(\mathcal C_\gamma\) is unbounded with positive probability.
Indeed, if the zero-cell were bounded almost surely, then it would follow for any countable dense set $M\subseteq \H^d$ that 
\[ \P(\text{some cell of $\mathcal{V}_\gamma$ unbounded}) \leq \sum_{x \in M} \P(\mathcal{C}_\gamma(x)\text{ unbounded}) = \sum_{x \in M} \P(\mathcal{C}_\gamma\text{ unbounded}) = 0, \]
where we used the fact that any cell contains a point of $M$ in the first step and the fact that $\P(\mathcal C_\gamma(x) \text{ unbounded})$ is independent of $x$ by isometry-invariance in the second step. 

We start by showing that for sufficiently large $r>0$ the set \(\mathcal C_\gamma \setminus B(o,r)\) consists of at least \(2^d\) unbounded connected components with positive probability. Then we use this to prove that with positive probability \(B(o,1)\) contains an encounter point.

In this section we will work with the Klein model of $\H^d$ and identify $\partial\H^d$ with the unit sphere $\mathbb{S}^{d-1}$. Let us start with some technical preparations. For any geodesic ray $R = [o,u)$ starting at $o$ denote by $\theta(R) \coloneqq u \in\mathbb{S}^{d-1}$ the unique ideal point belonging to $R$. Further denote by 
\[
\capop(z,\alpha) :=\{u\in\mathbb{S}^{d-1}\colon {\rm arccos}\langle u,z\rangle <\alpha\}
 =  \{u\in\mathbb{S}^{d-1}\colon d_{\mathbb{S}^{d-1}}(u,z) <\alpha\}, \quad z\in\mathbb{S}^{d-1},
\]
an open spherical cap with center $z$ and spherical radius $\alpha\in (0,\pi]$. 


\begin{lemma}\label{lm:inf_rays}
    Let $\varepsilon \in (0,\arccos(1-2/d)/2)$ and let $\gamma > 0$ be fixed such that with positive probability \(\mathcal{V}_\gamma\) contains an unbounded cell. For any $\boldsymbol\delta = (\delta_1,\ldots,\delta_d)\in\{-1,1\}^d$ we define
    \begin{equation}\label{eq:cap_centers}
    v_{\boldsymbol\delta}:={1\over \sqrt{d}}(\delta_1,\ldots,\delta_d)\in \mathbb{S}^{d-1}.
    \end{equation}
    Then with positive probability there are $2^d$ geodesic rays $R_{\boldsymbol\delta}$, $\boldsymbol\delta\in\{-1,1\}^d$, starting at $o$, contained in $\mathcal C_\gamma$ and such that $\theta(R_{\boldsymbol\delta})\in {\rm cap}(v_{\boldsymbol\delta},\varepsilon)$.
    Moreover, the $2^d$ sets $\capop(v_{\boldsymbol\delta},\varepsilon)$ are pairwise disjoint.
\end{lemma}

\begin{proof}
    Under the assumptions of the lemma, we already observed that there is a positive probability that the zero-cell is unbounded and thus contains a geodesic ray.
    Note that the unit sphere is covered by finitely many spherical caps with spherical radius $\varepsilon$. Therefore, there exists $v_0\in\mathbb S^{d-1}$ such that $\mathbb P(\mathcal A(v_0,\varepsilon))>0$, where we denote by $\mathcal A(v,\varepsilon)$, $v\in\mathbb S^{d-1}$, the event that there exists a geodesic ray $R$ starting at $o$ contained in $\mathcal C_\gamma$ and satisfying $\theta(R)\in{\rm cap}(v,\varepsilon)$. Notice that $\mathcal A(v,\varepsilon)$ is a decreasing event for the Poisson process $\xi_\gamma$ (in the sense of \cite[Section~20.3]{LP17}). Since (Euclidean) rotations with center $o$ are isometries in the Klein model, we obtain that 
    $$
    \inf_{\boldsymbol\delta\in\{-1,1\}^d}  \mathbb P(\mathcal A(v_{\boldsymbol\delta},\varepsilon)) = \mathbb P(\mathcal A(v_0,\varepsilon))>0.
    $$
    The FKG inequality for the Poisson process (see \cite[Theorem 20.4]{LP17}) implies that 
    $$
    \mathbb P\bigg(\bigcap_{\boldsymbol\delta \in \{-1,1\}^d}  \mathcal A(v_{\boldsymbol\delta},\varepsilon)\bigg)\ge \mathbb P(\mathcal A(v_0,\varepsilon))^{2^d}>0,  
    $$
    which proves the first assertion.
    
    Since $d_{\mathbb{S}^{d-1}}(v_{\boldsymbol\delta},v_{\boldsymbol\delta'}) \geq \arccos(1-2/d)$ for  distinct $\boldsymbol\delta,\boldsymbol\delta' \in \{-1,1\}^d$, the sets $\capop(v_{\boldsymbol\delta},\varepsilon)$, $\boldsymbol\delta \in \{-1,1\}^d$, are pairwise disjoint.
\end{proof}




\begin{lemma}\label{lm:zero_is_encounter_point}
    Let $r > r_0(d)=2\,{\rm artanh}(\sqrt{1-1/d})$ and let $\gamma > 0$ be fixed such that such that with positive probability \(\mathcal{V}_\gamma\) contains an unbounded cell. Then with positive probability $\mathcal C_\gamma\setminus B(o,r)$ has at least $2^d$ unbounded connected components.
\end{lemma}
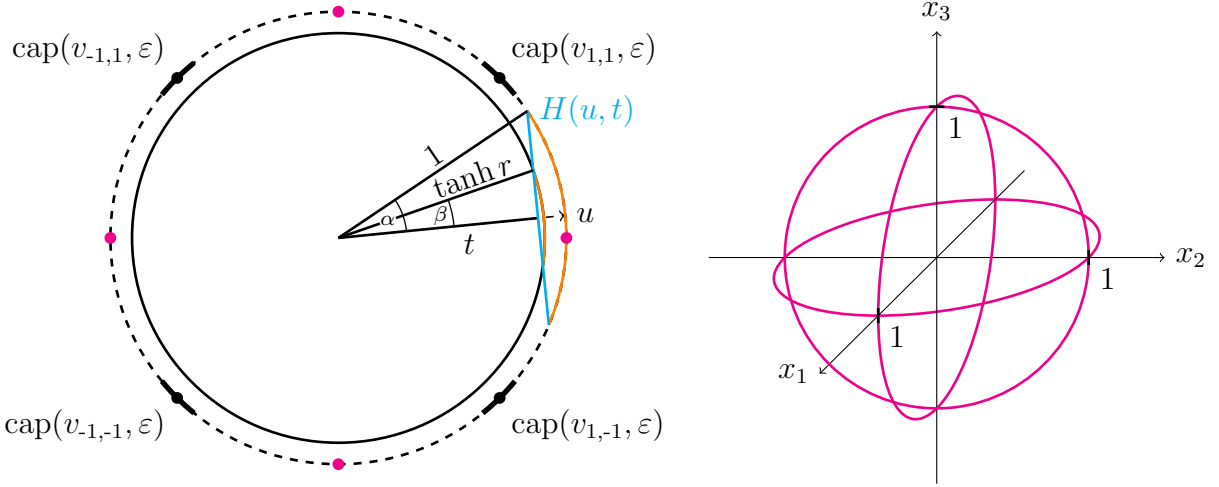
\begin{figure}
\begin{minipage}{.563\textwidth} 
        \centering
\begin{tikzpicture}[scale=3]
\draw[line width=1pt, dashed] (0,0) circle (1);
\draw[line width=1pt] (0,0) circle (0.9051482536448665);
\draw[line width=.6pt, ->, dashed] (0,0) -- (0.9950041652780258,0.09983341664682815) node[right] {$u$};
\draw[line width=1pt] (0,0) -- (0.8756036654446627,0.08785340664920878) node[midway, below=-2pt] {$\qquad t$};
\draw [color=orange, line width=1pt,domain=-22.628058625019747:34.087214527636206] plot ({ 0+1*cos(\x)}, {0+1*sin(\x)});
\draw [color=orange, line width=1pt,domain=-7.808043620743848:19.267199523360315] plot ({ 0+0.9051482536448665*cos(\x)}, {0+0.9051482536448665*sin(\x)});
\draw[line width=1pt,color=cyan] (0.9230219110839516,-0.38474738681234694) -- (0.8281854198053737,0.5604542001107645) 
              node[right] {$H(u,t)$};
\draw[line width=1pt] (0,0) -- (0.8281854198053737,0.5604542001107645) node[midway,above=-3pt,rotate=34.087214527636206] {$\quad 1$};
\draw[line width=1pt] (0,0) -- (0.8544509077230265,0.2986754214321088) node[midway,above=-3pt,rotate=19.267199523360315] {$\quad\qquad\tanh r$};
\draw[line width=1pt, fill=black, color=black] (0.7071067811865476,0.7071067811865476) circle (.02)
            node[anchor=south west] {${\rm cap} (v_{1,1},\varepsilon)$};
\draw[line width=1pt, fill=black, color=black] (0.7071067811865476,-0.7071067811865476) circle (.02)
                node[anchor=north west] {${\rm cap} (v_{1,\text{-}1},\varepsilon)$};
\draw[line width=1pt, fill=black, color=black] (-0.7071067811865476,-0.7071067811865476) circle (.02)
                node[anchor=north east] {${\rm cap} (v_{\text{-}1,\text{-}1},\varepsilon)$};
\draw[line width=1pt, fill=black, color=black] (-0.7071067811865476,0.7071067811865476) circle (.02)
            node[anchor=south east] {${\rm cap} (v_{\text{-}1,1},\varepsilon)$};
\draw [black, line width=2pt,domain=39.27042204869177:50.72957795130824] plot ({ 0+1*cos(\x)}, {0+1*sin(\x)});
\draw [black, line width=2pt,domain=129.27042204869176:140.72957795130824] plot ({ 0+1*cos(\x)}, {0+1*sin(\x)});
\draw [black, line width=2pt,domain=219.27042204869176:230.72957795130822] plot ({ 0+1*cos(\x)}, {0+1*sin(\x)});
\draw [black, line width=2pt,domain=309.2704220486918:320.7295779513082] plot ({ 0+1*cos(\x)}, {0+1*sin(\x)});
\draw[line width=1pt, magenta, fill=magenta] (1,0) circle (.02);
\draw[line width=1pt, magenta, fill=magenta] (0,1) circle (.02);
\draw[line width=1pt, magenta, fill=magenta] (-1,0) circle (.02);
\draw[line width=1pt, magenta, fill=magenta] (0,-1) circle (.02);
\draw [line width= .5pt,domain=5.729577951308232:34.087214527636206] plot ({ 0+0.3*cos(\x)}, {0+0.3*sin(\x)});
\draw [line width= .5pt,domain=5.729577951308232:19.267199523360315] plot ({ 0+0.51*cos(\x)}, {0+0.51*sin(\x)});
\node[fill=white,circle, inner sep=0pt] at (0.21625479451146884,0.07831898780501695) {$\scriptstyle\alpha$};
\node[] at (0.4490989629703403,0.0995495929623266) {$\scriptstyle \beta$};
\end{tikzpicture}
\end{minipage}
\hspace{.01\textwidth}
\begin{minipage}{.2\textwidth} 
    \centering
\begin{tikzpicture}[scale=2]
\begin{scope}[canvas is zx plane at y=0]
    \draw[line width=1pt, magenta] (0,0) circle (1cm);
    \draw[->] (0,-1.5) -- (0,1.5);
    \node[right] at (0,1.5) {$x_2$};
  \end{scope}

\begin{scope}[canvas is zy plane at x=0]
    \draw[line width=1pt, magenta] (0,0) circle (1cm);
    \draw[->] (-1.5,0) -- (2,0);
    \node[left] at (2,0) {$x_1$};
    \draw[line width=1pt] (1,-.05) -- (1,.05);
    \node[anchor = north west] at (1,0) {$1$};
    \draw[->] (0,-1.5) -- (0,1.5);
    \node[above] at (0,1.5) {$x_3$};
  \end{scope}

\begin{scope}[canvas is xy plane at z=0]
    \draw[line width=1pt, magenta] (0,0) circle (1cm);
    \draw[line width=1pt] (1,-.05) -- (1,.05);
    \node[anchor = north west] at (1,0) {$1$};
    \draw[line width=1pt] (-.05,1) -- (.05,1);
    \node[anchor = north west] at (0,1) {$1$};
  \end{scope}
\end{tikzpicture}
\end{minipage}

\caption{\textbf{Left}: Sketch of the situation in dimension $d=2$.
Pictured is a hyperplane $H(u,t)$ (cyan) and the caps that it `cuts off' from $\partial B_{\R^d}(o,1)$ and $\partial B_{\R^d}(o,\tanh(r))$ respectively (orange).
The opening angle $\alpha$ of the former satisfies $\cos(\alpha) = t$, while the opening angle $\beta$ of the latter satisfies $\cos(\beta) = t/\tanh(r)$.
Note that $\tanh(r)$ needs to be larger than $1/\sqrt{2}$ in order for the orange caps to be disjoint from the black caps.
In dimension $2$, the set $G$ consists of the four points $\{(\pm 1,0),(0,\pm 1)\}$ (magenta).
\textbf{Right}: The set $G$ (magenta) in dimension $d=3$.
In general, the set $G$ has dimension $d-2$.
\label{fig:sketch_caps}}
\end{figure}
\begin{proof} We again work in the Klein model. By \eqref{eq:KleinHyperplaneModel}, the Poisson hyperplane process with intensity $\gamma$ is identified with a Poisson process $\xi_\gamma^\text{Kl}$, with intensity measure $\gamma \cdot \mu_{d-1}^\text{Kl}$, of non-empty intersections $B_{\R^d}(o,1) \cap H'$, where $H'$ is a Euclidean hyperplane. The zero-cell is identified with the set
$$
\mathcal{C}^\text{Kl}_\gamma = \{x \in B_{\R^d}(o,1) \colon \text{$[o,x] \cap H = \emptyset$ for all $H \in \xi_\gamma^\text{Kl}$}\}.
$$
Recall that in the Klein model, a subset $M \subset B_{\R^d}(o,1)$ is unbounded if and only if its closure with respect to the usual Euclidean metric intersects $\partial B_{\R^d}(o,1)$. We also recall from~\eqref{eq:DistanceKlein} that $B_{\H^d}^\text{Kl}(o,r) = B_{\R^d}(o,\tanh(r))$.

Now let $r>r_0(d)$ be fixed. Let $\varepsilon\in(0,\arccos(1-2/d)/2)$. We denote by $\mathcal{E}_1 = \mathcal{E}_1(\xi_\gamma^\text{Kl})$ the event that for each $\boldsymbol\delta \in \{-1,1\}^d$, there exists a ray $R_{\boldsymbol\delta} = [0,x_{\boldsymbol\delta}) \subset B_{\R^d}(o,1)$ such that $R_{\boldsymbol\delta} \subset \mathcal{C}_\gamma^\text{Kl}$ and $\theta(R_{\boldsymbol\delta}) =  x_{\boldsymbol\delta} \in \capop(v_{\boldsymbol\delta},\varepsilon)$. We denote by $\mathcal{E}_2=\mathcal{E}_2(\xi_\gamma^\text{Kl})$ the event that $\mathcal C^\text{Kl}_\gamma\setminus B_{\H^d}^\text{Kl}(o,r)$ has at least $2^d$ unbounded connected components. 

The strategy of the proof is then as follows. By Lemma~\ref{lm:inf_rays}, the event $\mathcal E_1=\mathcal E_1(\xi_\gamma^\text{Kl})$ holds with positive probability. We will modify $\xi_\gamma^\text{Kl}$ by adding hyperplanes in such a way that there are ``walls" of hyperplanes which disconnect the sets ${\rm cap}(v_{\boldsymbol\delta},\varepsilon)$ in $B_{\mathbb R^d}(o,1)\setminus B_{\R^d}(o,\tanh(r))$ and, moreover, none of the additional hyperplanes intersects any of the rays from the origin to these caps. If $\xi_{\gamma,\text{mod}}^\text{Kl}$ denotes the modified process, this ensures that on the event $\mathcal E_1(\xi_{\gamma}^\text{Kl})$ the event $\mathcal E_2(\xi_{\gamma,\text{mod}}^\text{Kl})$ holds. Thus if the modification is also such that $\xi_{\gamma,\text{mod}}^\text{Kl}$ is absolutely continuous with respect to $\xi_{\gamma}^\text{Kl}$, it follows that $\mathcal E_2(\xi_{\gamma}^\text{Kl})$ holds with positive probability. We will show that such a modification can be achieved by adding uniformly drawn hyperplanes from finitely many sets, which thanks to insertion tolerance of the Poisson process will allow us to conclude the proof.

We now provide the details.
Write $H(u,t) \coloneqq H'(u,t) \cap B_{\R^d}(o,1)$, where $H'(u,t)$ is a Euclidean hyperplane with normal vector $u \in \mathbb{S}^{d-1}$ and Euclidean distance $t \in [0,1)$ in direction $u$ from the origin.

The set of ideal points $w \in \mathbb{S}^{d-1}$ that are ``cut off from $o$" by adding $H(u,t)$ form a spherical cap with spherical radius $\arccos(t)$, more precisely,
    \begin{equation}\label{eq:cap_radius_1}
         \{w \in \mathbb{S}^{d-1} \colon [o,w) \cap H(u,t) \neq \emptyset\} = \capop(u,\arccos(t)).
    \end{equation}
    For $t < \tanh(r)$, the set of points in $\partial B_{\R^d}(o,\tanh(r))$ that are ``cut off from $o$" form a spherical cap as well, namely
    \begin{equation}\label{eq:cap_radius_r}
    \{w \in \mathbb{S}^{d-1} \colon [o,\tanh(r)\cdot w) \cap H(u,t) \neq \emptyset\} = \capop(u,\arccos({t}/{\tanh(r)})),
    \end{equation}
    compare Figure \ref{fig:sketch_caps}.
    Define
    \[ G \coloneqq \{u \in \mathbb{S}^{d-1} \colon \text{$u_i=0$ for some $i=1,\ldots,d$} \}, \]
    see Figure \ref{fig:sketch_caps}.
    Since $G$ is compact, there exists a finite cover $G \subset \bigcup_{i=1,\ldots,k} \capop(u_i,\varepsilon)$ with $k\in\N$ and $u_1,\ldots,u_k \in G$. Fix constants $a,b$ with $\tanh(r_0) < a < b < \tanh(r)$ and define sets of hyperplanes
    \[ A_i \coloneqq \{ H(u,t) \colon\, d_{\mathbb{S}^{d-1}}(u,u_i) < \varepsilon, a < t < b \}, \quad i=1,\ldots,k. \]
    Note that $\mu_{d-1}^\text{Kl}(A_i)\in(0,\infty)$ for all $i = 1,\ldots,k$.
    Independently for each $i = 1,\ldots,k$, choose $H_i$ uniformly at random with distribution $\mu_{d-1}^\text{Kl}|_{A_i}$.
    Set 
    $$
    \xi_{\gamma,\text{mod}}^\text{Kl}\coloneqq \xi_\gamma^\text{Kl} \cup \{H_1,\ldots,H_k\}.
    $$ 
    Then $\xi_{\gamma,\text{mod}}^\text{Kl}$ is absolutely continuous in law with respect to $\xi_\gamma^\text{Kl}$ by insertion tolerance of the Poisson process (see~\cite[Theorem 1.4 \& Remark 1.6]{HS13}). We write $\mathcal{C}_{\gamma,\text{mod}}^\text{Kl}$ for the zero-cell associated to $\xi_{\gamma,\text{mod}}^\text{Kl}$.

    We now show that when $\varepsilon$ is chosen sufficiently small, the added hyperplanes cover $\tanh(r)\cdot G$ in the sense that
    \begin{equation}\label{eq:cap_1}
        [o,\tanh(r)\cdot u) \cap (H_1 \cup \cdots\cup H_k) \neq \emptyset, \quad \text{for all $u \in G$},
    \end{equation}
    and that further
    \begin{equation}\label{eq:cap_2}
        [o,v) \cap (H_1 \cup \cdots\cup H_k) = \emptyset, \quad \text{for all $v \in \capop(v_{\boldsymbol\delta},\varepsilon)$},\, \boldsymbol\delta\in \{-1,1\}^d.
    \end{equation}
    Note that on the event $\mathcal{E}_1(\xi_\gamma^\text{Kl})$, there exist rays $R_{\boldsymbol\delta} \subset \mathcal{C}_\gamma^\text{Kl}$, $\boldsymbol\delta \in \{-1,1\}^d$, with respective endpoints $\theta(R_{\boldsymbol\delta}) \in \capop(v_{\boldsymbol\delta},\varepsilon)$.
    Condition \eqref{eq:cap_2} ensures that these rays are contained in $\mathcal{C}_{\gamma,\text{mod}}^\text{Kl}$ and condition \eqref{eq:cap_1} guarantees that they intersect different connected components of $\mathcal{C}_{\gamma,\text{mod}}^\text{Kl} \setminus B_{\R^d}(o,\tanh(r))$. Thus to complete the proof of the lemma, it only remains to verify \eqref{eq:cap_1} and \eqref{eq:cap_2} for small $\varepsilon$.
    
    For $i = 1,\ldots,k$, it follows from \eqref{eq:cap_radius_1} and $H_i \in A_i$ that
    \[ \{w \in \mathbb{S}^{d-1} \colon [o,w) \cap H_i \neq \emptyset\} 
    \subset \capop(u_i,\arccos(a)+\varepsilon). \]
    Since
    \[ d_{\mathbb{S}^{d-1}}(u,v_{\boldsymbol\delta}) \geq \arccos(\sqrt{1-1/d}) > \arccos(a), \quad \boldsymbol\delta \in \{-1,1\}^d,\, u \in G, \]
    it follows that \eqref{eq:cap_2} is met as long as $\varepsilon$ is smaller than some constant $\varepsilon_0 = \varepsilon_0(a,d)$.

    On the other hand, \eqref{eq:cap_radius_r} and $H_i \in A_i$ yield that
    \[ \{w \in \mathbb{S}^{d-1} \colon [o,\tanh(r)\cdot w) \cap H_i \neq \emptyset\} \supset \capop(u_i,\arccos({b}/{\tanh(r)})-\varepsilon), \]
    for $i=1,\ldots,k$.
    Since $\arccos(b/\tanh(r))>0$, it follows that
    \[ \capop(u_i,\arccos({b}/{\tanh(r)})-\varepsilon) \supset \capop(u_i,\varepsilon), \]
    as long as $\varepsilon$ is smaller than the constant $\varepsilon_1 \coloneqq \arccos(b/\tanh(r))/2$, which implies \eqref{eq:cap_1} because $G$ is covered by the caps $\capop(u_i,\varepsilon)$.
    This concludes the proof.
\end{proof}

\begin{proof}[Proof of Lemma \ref{lem:encounter_pos_prob}]


Let $N$ be the number of encounter points $y$ in $B(o,1)$. Consider
\begin{align*}
 \E[N]&=\E\Big[\sum_{y\in\mathcal{Y}\cap B(o,1)}\one\{\mathcal{C}_{\gamma}(y)\setminus B(y,r)\text{ has $\ge 3$ unbounded connected components}\}\\
 &\qquad\qquad\times \one\{B(y,2r)\cap (\mathcal{Y}\setminus \{y\})=\emptyset\}\Big],
 \end{align*}
 where the expectation is taken with respect to $\mathcal{V}_{\gamma}$ and with respect to $\mathcal{Y}$. Since $\mathcal{Y}$ and $\mathcal{V}_{\gamma}$ are independent, and using Fubini's theorem and Mecke's formula (for $\mathcal Y$) we get
 \begin{align*}
 \E[N]&=\int_{B(o,1)}\P[\mathcal{C}_{\gamma}(y)\setminus B(y,r)\text{ has $\ge 3$ unbounded connected components}]\\
 &\hspace{3cm}\times\P[B(y,2r)\cap \mathcal{Y}=\emptyset]{\rm vol}(\dint y)\\
 &= {\rm vol}(B(o,1))\exp\big(-{\rm vol}(B(o,2r))\big)\\
 &\hspace{3cm}\times\P[\mathcal{C}_{\gamma}\setminus B(o,r)\text{ has $\ge 3$ unbounded connected components}]>0,
\end{align*}
where the second equality holds due to isometry-invariance of $\mathcal{V}_{\gamma}$ and $\mathcal{Y}$, and the last inequality holds by Lemma \ref{lm:zero_is_encounter_point}. Since $N$ is non-negative discrete random variable it follows that the event $\{N\ge 1\}$ has positive probability and the proof follows.
\end{proof}

\subsection{Phase transition}

We now prove the description of the phase transition for Poisson hyperplane percolation announced in Corollary~\ref{cor:phase}. One direction follows from Theorem~\ref{thm:main}. We give two proofs of the other direction; the first proof is a variation of standard arguments in discrete percolation theory, while the second proof is inspired by \cite[Section~4]{BS01} and \cite[Theorem~8.49]{LP16}.

\begin{proof}[First proof of Corollary \ref{cor:phase}] Let $d\ge2$, $\gamma>0$ and let $N_\gamma$ denote the number of unbounded cells in $\mathcal V_\gamma$.

\medskip

{\noindent\em Subcritical part.} Theorem~\ref{thm:main} implies that $N_{\gamma_c}=0$ almost surely. By the definition of  $\gamma_c$, $N_\gamma=0$ almost surely if $\gamma\ge\gamma_c$.

\medskip

{\noindent\em Supercritical part.} If $\gamma<\gamma_c$, the definition of $\gamma_c$ and monotonicity imply that $N_\gamma\ge1$ with positive probability. By ergodicity, there exists $k\in\{1,2,\ldots\}\cup\{\infty\}$ such that $N_\gamma=k$ almost surely. 

Suppose that $k<\infty$. It then follows that there exists $R>1$ such that $B(o,R)$ intersects all unbounded cells with positive probability. Recall from \cite{HS13} that the Poisson process $\xi:=\xi_\gamma$ is deletion tolerant. Since the number of elements of $\xi$ intersecting $B(o,R)$ is finite almost surely, the process $\xi'$ obtained by removing all of these elements is absolutely continuous with respect to $\xi$. On the event that $B(o,R)$ intersects all unbounded cells, the vacant set $\mathcal V'$ associated to $\xi'$ has a unique unbounded cell because each of the $\mathcal V'$-unbounded cells contains an unbounded $\mathcal V_\gamma$-cell and all unbounded $\mathcal V_\gamma$-cells are contained in the same unbounded $\mathcal V'$-cell by construction. Thus $\mathcal V'$ has a unique unbounded cell with positive probability. Absolute continuity implies that $\mathcal V_\gamma$ has a unique unbounded cell with positive probability, hence almost surely.
In other words, it holds that \(k=1\).
This leads to a contradiction as follows.

It follows from the Crofton formula that the two-point function of $\mathcal V_\gamma$ is given by $\tau_\gamma(w,z):=\mathbb P(\mathcal C_\gamma(w)=\mathcal C_\gamma(z)) = \exp(-\gamma d_{\mathbb H^d}(w,z)),$
see~e.g.~\cite{MR23}. Assuming the existence of a unique unbounded cell with positive probability, the FKG inequality for the Poisson process (see \cite[Theorem 20.4]{LP17}) implies
\begin{align*}
\inf_{w,z\in\mathbb H^d} \mathbb P( \mathcal C_\gamma(w)=\mathcal C_\gamma(z)) & \ge \inf_{w,z\in\mathbb H^d} \mathbb P( \mathcal C_\gamma(w) \text{ and } \mathcal C_\gamma(z) \text{ are unbounded}) \\
& \ge \, \, \mathbb P( \mathcal C_\gamma(o) \text{ is unbounded})^2 >0,
\end{align*}
which contradicts the fact that $\tau_\gamma(w,z)\to0$ as $d(w,z)\to\infty$.
\end{proof}

Here is the second, perhaps more geometric, proof.

\begin{lemma}[Limit points] \label{lm:limitpoins} Let $d\ge2$, $\gamma>0$ and let $\mathcal L_\gamma$ denote the set of points $z\in\partial \mathbb H^d$ such that there is a path in $\mathcal V_\gamma$ with limit $z$. Then almost surely $\mathcal L_\gamma=\varnothing$ or $\mathcal L_\gamma \subset\partial \mathbb H^d$ is dense.
\end{lemma}

In fact, this holds more generally. 

\begin{lemma}
Let \(\mathcal{Z}\) be a random closed subset of \(\H^d\), \(d \geq 2\), whose law is invariant under isometries of \(\H^d\).
Let \(\mathcal{V} \coloneqq \H^d \setminus \mathcal{Z}\) and define \(\mathcal{L}\) as the set of all points \(x \in \partial\H^d\) such that there exists a path in \(\mathcal{V}\) with limit \(x\).
The almost surely \(\mathcal{L} = \emptyset\) or \(\mathcal{L}\subset\partial \mathbb H^d\) is dense.
\end{lemma}
\begin{proof}
The proof follows roughly the lines of \cite[Lemma 4.3]{BS01}.
We work in the Klein model, so $\H^d$ is identified with the unit ball $B_{\R^d}(o,1)$ and $\partial \H^d$ with the sphere $\partial B_{\R^d}(o,1) = \mathbb{S}^{d-1}$.

Recall that the hyperplanes in this model are the nonempty intersections of \(B_{\R^d}(o,1)\) with Euclidean hyperplanes.
Consequently, hyperbolic half-spaces in this model are precisely intersections of the unit ball with Euclidean half-spaces whose limiting hyperplane intersects $B_{\R^d}(o,1)$.
The ideal limit points in $\partial \H^d$ of such a half-space appear as a closed spherical cap in the Klein model, i.e., they are of the form $\overline \capop(x,\alpha)$ for some \(x \in \mathbb{S}^{d-1}\) and \(\alpha \in (0,\pi)\).

Since any half-space can be mapped to any other half-space by an isometry of \(\H^d\), it follows that for any \(x_1,x_2 \in \mathbb{S}^{d-1}\) and \(\alpha_1,\alpha_2 \in (0,\pi)\) there exists an isometry \(g \in \Isom(\H^d)\) such that the action of \(g\) on \(\mathbb S^{d-1}\) maps \(\capop(x_1,\alpha_1)\) to \(\capop(x_2,\alpha_2)\).

Suppose now that \(\P(\mathcal L \neq \emptyset, \mathcal L \text{ not dense in }\partial\H^d) > 0\).
Conditioned on this event, \(\mathcal{X} \coloneqq \overline{\mathcal{L}}\) is a random closed subset of \(\mathbb{S}^{d-1}\), \(\mathcal X \notin \{\emptyset,\mathbb S^{d-1}\}\), and the distribution of \(\mathcal X\) is invariant under the actions of all \(g \in \Isom(\H^d)\).

It follows from \(\mathcal{X} \neq \mathbb S^{d-1}\) that there must be some \(x \in \mathbb{S}^{d-1}\) and \(\alpha \in (0,\pi)\) such that
\[ \P(\mathcal{X} \cap \capop(x,\alpha) = \emptyset) \eqqcolon \delta > 0. \]
By \(\Isom(\H^d)\)-invariance, this probability does not depend on the choice of \(x\) and \(\alpha\).

Let \(x,y \in \mathbb{S}^{d-1}\), \(y \notin \{x,-x\}\).
Choose \(\alpha \in (0,\pi)\) and \(\varepsilon\in(0,\pi/2)\) so that \(\capop(x,\varepsilon) \cup \capop(-x,\varepsilon) \subseteq  \capop(y,\alpha)\).
Recall that \(\mathcal X\) is not empty, so whenever \(\mathcal X \cap \capop(-x,\varepsilon) = \emptyset\), we have that \(\mathcal X\) intersects \(\capop(x,\pi-\varepsilon/2)\).
It follows that
\begin{align*}
    \delta
    &= \P(\mathcal X \cap \capop(y,\alpha) = \emptyset)\\
    &\leq \P(\mathcal X \cap \capop(-x,\varepsilon) = \emptyset, \mathcal X \cap \capop(x,\varepsilon) = \emptyset)\\
    &\leq \P(\mathcal X \cap \capop(x,\pi-\varepsilon/2) \neq \emptyset, \mathcal X \cap \capop(x,\varepsilon) = \emptyset)\\
    &= \P(\mathcal X \cap \capop(x,\pi-\varepsilon/2) \neq \emptyset) - \P(\mathcal X \cap \capop(x,\varepsilon) \neq \emptyset)\\
    &= (1-\delta) - (1-\delta) = 0,
\end{align*}
a contradiction. 

\end{proof}

\begin{lemma}[Percolation in half-spaces] \label{lm:Localization} Let $d\ge2$ and let $0<\gamma<\gamma_c$. Then for every half-space $W\subset\mathbb H^d$, $W\cap \mathcal V_\gamma$ has unbounded cells almost surely. 
\end{lemma}

\begin{proof} By definition of $\gamma_c$, there exists an unbounded cell with positive probability if $\gamma<\gamma_c$. By ergodicity, this holds almost surely. Convexity of cells implies that every unbounded cell contains a geodesic ray and hence almost surely the set $\mathcal L_\gamma$ of ideal limit points is not empty. Lemma \ref{lm:limitpoins} implies that $\mathcal L_\gamma\subset\partial \mathbb H^d$ is dense. The claim follows because the set of ideal limit points which can only be reached through $W$ has non-empty interior.
\end{proof}

The same proof shows the following.

\begin{lemma}
Let \(\mathcal{Z}\) be a random closed subset of \(\H^d\), \(d \geq 2\), whose law is invariant under isometries of \(\H^d\).
Let \(\mathcal{V} \coloneqq \H^d \setminus \mathcal{Z}\) and define \(\mathcal{L}\) as the set of all points \(x \in \partial\H^d\) such that there exists a path in \(\mathcal{V}\) with limit \(x\).
If \(\mathcal{V}\) has unbounded connected components a.s., then for every half-space $W\subset\mathbb H^d$, $W\cap \mathcal V$ has unbounded connected components.
\end{lemma}

\begin{proof}[Second proof of Corollary~\ref{cor:phase}] Let $k\in\N$ and choose half-spaces \(W_1,\ldots,W_k\) that are well separated in the sense that $d(W_i,W_j):=\inf \{d_{\mathbb H^d}(w,z) : w\in  W_i, z\in W_j\} > 0$ for $i\neq j$.
It follows that, with positive probability, each pair of half-spaces is separated by a hyperplane, i.e., for each \(1 \leq i < j \leq k\) there exists an \(H \in \xi\) so that \(W_i\) and \(W_j\) lie on different sides of \(H\).

On this event Lemma~\ref{lm:Localization} yields that a.s.\ \(N_\gamma \geq k\).
It follows that \(\mathbb P(N_\gamma\ge k)>0\) and, by ergodicity, \(\mathbb P(N_\gamma = \infty) = 1\).
\end{proof}

\section{Multiple phase transitions} \label{sec:MultiplePhases}

We now show that there are in fact $d-1$ distinct critical parameters where the system undergoes a phase transition.

\begin{proposition}\label{prop:multiple_phase}
    Let $d \geq 2$ and let $\xi_\gamma$ be a Poisson hyperplane process on $\H^d$ with intensity $\gamma$.
    For $k \in \{2,\ldots,d-1\}$ define
    \[ \gamma_c(d,k) \coloneqq \frac{k-1}{d-1} \gamma_c(\H^d). \]

    If $\gamma \geq \gamma_c(d,k)$, then a.s.\ all cells in the tessellation induced by $\xi_\gamma$ have bounded $k$-dimensional faces.

    If $\gamma < \gamma_c(d,k)$, then the following holds a.s.:
    Some cells in the tessellation induced by $\xi_\gamma$ have unbounded $k$-dimensional faces.
    Further, for any distinct hyperplanes $H_1,\ldots,H_{d-k} \in \xi_\gamma$, whose intersection $P = H_1 \cap \cdots\cap H_{d-k}$ is a $k$-dimensional plane, $P$ contains infinitely many unbounded $k$-dimensional faces.
\end{proposition}

\begin{remark}
    Note that a.s.~the intersection of any collection of $d-k$ distinct hyperplanes $H_1,\ldots,H_{d-k} \in \xi_\gamma$ is either empty or a $k$-dimensional plane (this follows from \cite[Lemma~2.3]{BHT23} and the multivariate Mecke formula).
    Hence, one can replace the condition $\dim(H_1\cap\cdots\cap H_{d-k})=k$ in the above remark by $H_1\cap\cdots\cap H_{d-k} \neq\emptyset$.
\end{remark}


We start with the following observation (recall that $\omega_j \coloneqq 2\pi^{j/2}/\Gamma(j/2)$ denotes the surface area of the Euclidean unit ball in $\R^j$):
\begin{lemma} \label{lem:IntersecionIntensity}
    Let $0<k<d$.
    $L \in A(d,k)$ be any $k$-plane in $\H^d$.
    For $\mu_{d-1}$-almost all $H \in A(d,d-1)$, the intersection $L \cap H$ is either empty or a $(k-1)$-plane in $\H^d$.
    Identifying $L \cong \H^k$, we get
    \begin{equation*}
        \int_{A(d,d-1)} \one\{H\cap L \neq \emptyset,H \cap L \in \cdot\} \mu_{d-1}(\dint H)
        = c(d,k) \int_{A(k,k-1)} \one\{H' \in \cdot\} \mu_{k-1}(\dint H'),
    \end{equation*}
    with $c(d,k) \coloneqq ({\omega_{d+1}\omega_{k}})/({\omega_{d}\omega_{k+1}})$.
\end{lemma}
\begin{proof}
    The proof follows essentially along the lines of \cite[Lemma 2.2 \& Lemma 2.3]{BHT23} and is thus omitted.
\end{proof}
    


Let $0<k<d$ and fix a $k$-plane $P \in A(d,k)$.
If $\xi_\gamma$ is a Poisson hyperplane process in $\H^d$ with intensity $\gamma$, then the mapping theorem for Poisson processes combined with Lemma~\ref{lem:IntersecionIntensity} implies the following:
Taking the process $\xi_\gamma^{(P)}$ of all non-trivial intersections of planes from $\xi_\gamma$ with $P$ yields again a Poisson process.
When we identify $P \cong \H^k$, this is a Poisson process with values in $A(k,k-1)$ and intensity measure $c(d,k)\cdot\gamma\cdot\mu_{k-1}$.

The process $\xi_\gamma^{(P)}$ induces a tessellation of $P$.
Let $N_\gamma(P) = N_\gamma(P,\xi_\gamma)$ denote the number of unbounded cells in this tessellation.
It follows directly from Corollary \ref{cor:phase} that $N_\gamma(P) = 0$ a.s.\ when $c(d,k)\cdot\gamma \geq \gamma_c(\H^k)$ and $N_\gamma(P) = \infty$ a.s.\ otherwise.

Hence, for $c(d,k)\cdot\gamma \geq \gamma_c(\H^k)$, the multivariate Mecke formula yields
\begin{align*}
    &\E\mathop{\sideset{}{^{\neq}}\sum}_{H_1,\ldots,H_{d-k}\in\xi_\gamma} \one\{N_\gamma(H_1\cap\cdots\cap H_{d-k})\neq 0\}
    \cdot \one\{\dim(H_1\cap\cdots\cap H_{d-k})=k\}\\
    &\quad= \int_{A(d,d-1)^{d-k}} \P(N_\gamma(H_1\cap\cdots\cap H_{d-k})\neq 0) 
    \cdot \one\{\dim(H_1\cap\cdots\cap H_{d-k})=k\}\\
    &\hspace*{10cm}\times\mu_{d-1}^{d-k}(\dint(H_1,\ldots,H_{d-k})) = 0.
\end{align*}
It follows that, almost surely, $N_\gamma(H_1\cap\cdots\cap H_{d-k}) = 0$ for all distinct $H_1,\ldots,H_{d-k} \in \xi_\gamma$ with $\dim(H_1\cap\cdots\cap H_{d-k})=k$.
Since any $k$-dimensional face of a cell in the original tessellation of $\H^d$ induced by $\xi_\gamma$ is a cell in the tessellation induced by $\xi_\gamma^{(H_1\cap\cdots\cap H_{d-k})}$ for some such choice of $H_1,\ldots,H_{d-k} \in \xi_\gamma$, it follows that a.s.\ all $k$-dimensional faces of cells are finite.

For $c(d,k)\cdot\gamma < \gamma_c(\H^k)$, replacing $\one\{N_\gamma(H_1\cap\cdots\cap H_{d-k})\neq 0\}$ by $\one\{N_\gamma(H_1\cap\cdots\cap H_{d-k}) < \infty\}$ in the above argument yields the following:
Almost surely, it holds for any distinct $H_1,\ldots,H_{d-k} \in \xi_\gamma$ with $\dim(H_1\cap\cdots\cap H_{d-k})=k$ that $N_\gamma(H_1\cap \cdots\cap H_{d-k}) = \infty$.
Hence, a.s.\ some cells in the tesselation induced by $\xi_\gamma$ have unbounded $k$-dimensional faces and any $k$-dimensional $P = H_1\cap \cdots \cap H_{d-k}$, with $H_1,\ldots,H_{d-k} \in \xi_\gamma$ distinct, contains infinitely many such faces.
 
Since $\omega_j = 2 \pi^{j/2}/\Gamma(j/2)$, $j\in\N$, it holds that
\[ c(d,k) = \frac{\omega_{d+1}\omega_{k}}{\omega_{d}\omega_{k+1}} = \frac{\Gamma(\frac{d}{2})\Gamma(\frac{k+1}{2})}{\Gamma(\frac{d+1}{2})\Gamma(\frac{k}{2})}, \]
and further 
\[ \gamma_c(\H^k) = \frac{(k-1)^2 \sqrt{\pi}\Gamma(\frac{k-1}{2})}{\Gamma(\frac{k}{2})} = \frac{2\sqrt{\pi}(k-1)\Gamma(\frac{k+1}{2})}{\Gamma(\frac{k}{2})}. \]
It thus follows for $2 \leq k \leq d-1$ that
\[ \gamma_c(d,k) = \frac{\gamma_c(\H^k)}{c(d,k)} = \frac{2\sqrt{\pi}(k-1)\Gamma(\frac{k+1}{2})}{\Gamma(\frac{k}{2})} \frac{\Gamma(\frac{d+1}{2})\Gamma(\frac{k}{2})}{\Gamma(\frac{d}{2})\Gamma(\frac{k+1}{2})}  = \frac{k-1}{d-1} \gamma_c(\H^d). \]
Proposition \ref{prop:multiple_phase} now follows readily.


\begin{remark}
    If we choose the normalization of the invariant measure $\mu_{d-1}$ differently by setting $\widetilde\mu_{j-1} \coloneqq \frac{\omega_j}{\omega_{j+1}} \mu_{j-1}$ for $j \geq 2$, the constants in Lemma \ref{lem:IntersecionIntensity} become $\widetilde c(d,k) = 1$ and the critical intensities
    read $\widetilde\gamma_c(\H^d) = 2\pi(d-1)$, $\widetilde\gamma_c(d,k) = 2\pi(k-1) = \widetilde\gamma_c(\H^k)$.
\end{remark}

\printbibliography

@phdthesis{Herold2021,
    author       = {Herold, Felix},
    year         = {2021},
    title        = {Random mosaics in hyperbolic space},
    doi          = {10.5445/IR/1000129986},
    publisher    = {{Karlsruher Institut für Technologie (KIT)}},
    pagetotal    = {195},
    school       = {Karlsruher Institut für Technologie (KIT)},
    language     = {english}
}

@article {Last10,
    AUTHOR = {Last, G\"unter},
     TITLE = {Stationary random measures on homogeneous spaces},
   JOURNAL = {J. Theoret. Probab.},
  FJOURNAL = {Journal of Theoretical Probability},
    VOLUME = {23},
      YEAR = {2010},
    NUMBER = {2},
     PAGES = {478--497},
      ISSN = {0894-9840,1572-9230},
   MRCLASS = {60D05 (60G55 60G57 60G60)},
  MRNUMBER = {2644871},
MRREVIEWER = {Christian\ Rau},
       DOI = {10.1007/s10959-009-0231-9},
       URL = {https://doi.org/10.1007/s10959-009-0231-9},
}

@article {BLPS99,
    AUTHOR = {Benjamini, Itai and Lyons, Russell and Peres, Yuval and
              Schramm, Oded},
     TITLE = {Critical percolation on any nonamenable group has no infinite
              clusters},
   JOURNAL = {Ann. Probab.},
  FJOURNAL = {The Annals of Probability},
    VOLUME = {27},
      YEAR = {1999},
    NUMBER = {3},
     PAGES = {1347--1356},
      ISSN = {0091-1798,2168-894X},
   MRCLASS = {60K35 (60B99)},
  MRNUMBER = {1733151},
MRREVIEWER = {Olle\ H\"aggstr\"om},
       DOI = {10.1214/aop/1022677450},
       URL = {https://doi.org/10.1214/aop/1022677450},
}

@article {BHT23,
    AUTHOR = {Betken, Carina and Hug, Daniel and Th\"ale, Christoph},
     TITLE = {Intersections of {P}oisson {$k$}-flats in constant curvature
              spaces},
   JOURNAL = {Stochastic Process. Appl.},
  FJOURNAL = {Stochastic Processes and their Applications},
    VOLUME = {165},
      YEAR = {2023},
     PAGES = {96--129},
      ISSN = {0304-4149,1879-209X},
   MRCLASS = {60D05 (52A22 52A55 53C65 60F05)},
  MRNUMBER = {4635692},
MRREVIEWER = {Anna\ Gusakova},
       DOI = {10.1016/j.spa.2023.08.001},
       URL = {https://doi.org/10.1016/j.spa.2023.08.001},
}

@article {HHT21,
    AUTHOR = {Herold, Felix and Hug, Daniel and Th\"ale, Christoph},
     TITLE = {Does a central limit theorem hold for the {$k$}-skeleton of
              {P}oisson hyperplanes in hyperbolic space?},
   JOURNAL = {Probab. Theory Related Fields},
  FJOURNAL = {Probability Theory and Related Fields},
    VOLUME = {179},
      YEAR = {2021},
    NUMBER = {3-4},
     PAGES = {889--968},
      ISSN = {0178-8051,1432-2064},
   MRCLASS = {60D05 (52A22 52A55 53C65 60F05)},
  MRNUMBER = {4242628},
       DOI = {10.1007/s00440-021-01032-w},
       URL = {https://doi.org/10.1007/s00440-021-01032-w},
}

@misc{BHT25,
      title={Intersection density and visibility for Boolean models in hyperbolic space}, 
      author={Tillmann Bühler and Daniel Hug and Christoph Thaele},
      year={2025},
      eprint={2501.13447},
      archivePrefix={arXiv},
      primaryClass={math.PR},
      url={https://arxiv.org/abs/2501.13447}, 
}

@book {Ra19,
    AUTHOR = {Ratcliffe, John G.},
     TITLE = {Foundations of hyperbolic manifolds},
    SERIES = {Graduate Texts in Mathematics},
    VOLUME = {149},
   EDITION = {Third},
 PUBLISHER = {Springer, Cham},
      YEAR = {2019},
     PAGES = {xii+800},
      ISBN = {978-3-030-31597-9},
   MRCLASS = {57M50 (20H10 30F40 57K32)},
  MRNUMBER = {4221225},
       DOI = {10.1007/978-3-030-31597-9},
       URL = {https://doi.org/10.1007/978-3-030-31597-9},
}

@book {HS,
    AUTHOR = {Hug, Daniel and Schneider, Rolf},
     TITLE = {Poisson hyperplane tessellations},
    SERIES = {Springer Monographs in Mathematics},
 PUBLISHER = {Springer, Cham},
      YEAR = {2024},
     PAGES = {xi+550},
      ISBN = {978-3-031-54104-9},
   MRCLASS = {60-02 (52A22 52C35 60D05 60G55)},
  MRNUMBER = {4807325},
       DOI = {10.1007/978-3-031-54104-9},
       URL = {https://doi.org/10.1007/978-3-031-54104-9},
}

@article {GKT,
    AUTHOR = {Godland, Thomas and Kabluchko, Zakhar and Th\"ale, Christoph},
     TITLE = {Beta-star polytopes and hyperbolic stochastic geometry},
   JOURNAL = {Adv. Math.},
  FJOURNAL = {Advances in Mathematics},
    VOLUME = {404},
      YEAR = {2022},
     PAGES = {Paper No. 108382, 69},
      ISSN = {0001-8708,1090-2082},
   MRCLASS = {52A22 (51M10 52A55 60D05 60F05)},
  MRNUMBER = {4406257},
MRREVIEWER = {Matthias\ Reitzner},
       DOI = {10.1016/j.aim.2022.108382},
       URL = {https://doi.org/10.1016/j.aim.2022.108382},
}

@article {BS01,
    AUTHOR = {Benjamini, Itai and Schramm, Oded},
    TITLE = {Percolation in the hyperbolic plane},
    JOURNAL = {J. Amer. Math. Soc.},
    FJOURNAL = {Journal of the American Mathematical Society},
    VOLUME = {14},
    NUMBER = {2},
      YEAR = {2000},
     PAGES = {487-507},
}

@article {HS13,
    AUTHOR = {Holroyd, Alexander and Soo, Terry},
    TITLE = {Insertion and deletion tolerance of point processes},
    JOURNAL = {Electron. J. Probab.},
    FJOURNAL = {Electronic Journal of Probability},
    VOLUME = {18},
    YEAR = {2013},
    PAGES = {1-24},
}

@book {LP16,
    AUTHOR = {Lyons, Russell and Peres, Yuval},
     TITLE = {Probability on Trees and Networks},
    SERIES = {Cambridge Series in Statistical and Probabilistic Mathematics},
    VOLUME = {42},
   EDITION = {First Edition},
 PUBLISHER = {Cambridge University Press},
      YEAR = {2016},
}

@article{BJST09,
    AUTHOR = {Benjamini, Itai and Jonasson, Johan and Schramm, Oded and Tykesson, Johan},
    TITLE = {Visibility to infinity in the hyperbolic plane,
despite obstacles},
    JOURNAL = {ALEA Lat. Am. J. Probab. Math. Stat.},
    FJOURNAL = {Latin American Journal of Probability and Mathematical Statistics},
    VOLUME = {6},
    YEAR = {2009},
    PAGES = {323-342},
}

@article{TC13,
    AUTHOR = {Tykesson, Johan and Calka, Pierre},
    TITLE = {Asymptotics of visibility in the hyperbolic plane},
    JOURNAL = {Adv. in Appl. Probab.},
    FJOURNAL = {Advances in Applied Probability},
    VOLUME = {45},
    NUMBER = {2},
    YEAR = {2013},
    PAGES = {332-350},
}

@article{P07,
    AUTHOR = {Porret-Blanc, Sylvain},
    TITLE = {Sur le caractère borné de la cellule de Crofton des mosaïques de géodésiques dans le plan hyperbolique},
    JOURNAL = {C. R. Acad. Sci. Paris, Ser. I},
    FJOURNAL = {Comptes Rendus. Mathématique},
    VOLUME = {344},
    NUMBER = {8},
    YEAR = {2007},
    PAGES = {477-481},
}

@misc{MS25,
      title={Indistinguishability of unbounded components in the occupied and vacant sets of Boolean models on symmetric spaces}, 
      author={Yingxin Mu and Artem Sapozhnikov},
      year={2025},
      eprint={2504.05027},
      archivePrefix={arXiv},
      primaryClass={math.PR},
      url={https://arxiv.org/abs/2504.05027}, 
}

@book {LP17,
    AUTHOR = {Last, G{\"u}nter and Penrose, Mathew},
     TITLE = {Lectures on the Poisson Process},
    SERIES = {Institute of Mathematical Statistics Textbooks},
    VOLUME = {7},
   EDITION = {},
 PUBLISHER = {Cambridge University Press},
      YEAR = {2017},
}

@misc{MR23,
      title={Haagerup property and group-invariant percolation}, 
      author={Chiranjib Mukherjee and Konstantin Recke},
      year={2023},
      eprint={2303.17429},
      archivePrefix={arXiv},
      primaryClass={math.GR},
      url={https://arxiv.org/abs/2303.17429}, 
}

@article{BK89,
author = {R. M. Burton and M. Keane},
title = {{Density and uniqueness in percolation}},
volume = {121},
journal ={Comm. Math. Phys.},
fjournal = {Communications in Mathematical Physics},
number = {3},
publisher = {Springer},
pages = {501 -- 505},
year = {1989},
}

@article{S72,
author = {Shepp, L. A.},
title = {Covering the circle with random arcs},
journal = {Israel J. Math.},
fjournal = {Israel Journal of Mathematics},
volume = {11},
pages = {328–345},
year = {1972},
}

@article{H73,
 ISSN = {00255521, 19031807},
 URL = {http://www.jstor.org/stable/24490572},
 author = {J{\o}rgen Hoffmann-J{\o}rgensen},
 journal = {Math. Scand.},
 fjournal = {Mathematica Scandinavica},
 number = {2},
 pages = {169--186},
 publisher = {Mathematica Scandinavica},
 title = {Coverings of metric spaces with randomly placed balls},
 urldate = {2025-11-14},
 volume = {32},
 year = {1973}
}

@article{FJJS18,
 author = {De-Jun Feng and Esa J{\"a}rvenp{\"a}{\"a} and Maarit J{\"a}rvenp{\"a}{\"a} and Ville Suomala},
 journal = {Ann. Probab.},
 fjournal = {Annals of Probability},
 number = {3},
 pages = {1542-1596},
 title = {Dimensions of random covering sets in Riemann manifolds},
 volume = {46},
 year = {2018}
}

@misc{HLS24,
      title={Boolean models in hyperbolic space}, 
      author={Daniel Hug and Günter Last and Matthias Schulte},
      year={2024},
      eprint={2408.03890},
      archivePrefix={arXiv},
      primaryClass={math.PR},
      url={https://arxiv.org/abs/2408.03890}, 
}

@article {BT15,
    AUTHOR = {Broman, Erik and Tykesson, Johan},
     TITLE = {Poisson cylinders in hyperbolic space},
   JOURNAL = {Electron. J. Probab.},
  FJOURNAL = {Electronic Journal of Probability},
    VOLUME = {20},
      YEAR = {2015},
     PAGES = {no. 41, 25},
      ISSN = {1083-6489},
   MRCLASS = {60K35 (60B99 82B26)},
  MRNUMBER = {3335832},
MRREVIEWER = {Jean-Claude\ Gruet},
       DOI = {10.1214/EJP.v20-3645},
       URL = {https://doi.org/10.1214/EJP.v20-3645},
}

\newpage
\appendix

\section{Ergodicity}

In this section, we show that homogeneous Poisson hyperplane processes are mixing and hence ergodic. We closely follow the treatment of the Euclidean case in \cite[Chapter~7]{HS} to which we refer the reader for further details.
A sketch for the proof of ergodicity in the hyperbolic setting can also be found in \cite[Proposition 2.1]{BT15}.

We introduce the notation \(\mathcal{F}'(\H^d)\) for the space of non-empty closed subsets of \(\H^d\), equipped with the Fell topology and the Borel \(\sigma\)-Algebra derived from it.
We further write \(N_s(\mathcal{F}'(\H^d))\) and \(\mathcal{N}_s(\mathcal{F}'(\H^d))\) for the space of simple counting measures on \(\mathcal{F}'(\H^d)\) and the \(\sigma\)-algebra generated by the evaluation maps \(\xi \mapsto \xi(A)\), \(A \in \mathcal{F}'(\H^d)\), respectively.

Let \(\Isom(\H^d)\) denote the isometry group of \(\H^d\) and \(\Isom^+(\H^d)\) the subgroup of orientation-preserving isometries.
We say that a sequence \(\varphi_n\) in \(\Isom(\H^d)\) goes to infinity and write \(\varphi_n \to \infty\) when \(d(x,\varphi_n(x)) \to \infty\) for some (and hence for all) \(x \in \H^d\).
(Note that this turns out to be equivalent to saying that every compact set in \(\Isom(\H^d)\), equipped with the compact-open topology, contains only finitely many \(\varphi_n\).)

We call a stationary simple point process \(\xi\) in \(\mathcal{F}'(\H^d)\) \emph{mixing} when
\begin{equation}\label{eq:mixing}
    \lim_{\substack{\varphi \in \Isom^+(\H^d),\\\varphi \to \infty}} \P(\xi \in (A \cap \varphi(B)) = \P(\xi \in A) \P(\xi \in B)
\end{equation}
holds for all \(A,B \in \mathcal{N}_s(\mathcal{F}'(\H^d))\).

\begin{theorem}
    \label{thm:mixing}
    Let \(\xi\) be a homogeneous Poisson process of hyperplanes in \(\H^d\) with intensity \(\gamma > 0\).
    When interpreted as a simple point process in \(\mathcal{F}'(\H^d)\) concentrated on the space of hyperplanes, \(\xi\) is mixing.
\end{theorem}

Note that mixing immediately implies \emph{ergodicity}:

\begin{corollary}
\label{cor:ergodic}
    Let \(\xi\) be as in Theorem \ref{thm:mixing}.
    Then \(\xi\) is ergodic in the sense that
    \begin{equation*}
        \P(\xi \in A) \in \{0,1\}
    \end{equation*}
    holds for any event \(A \in \mathcal{N}_s(\mathcal{F}'(\H^d))\) that is invariant under orientation-preserving isometries.
\end{corollary}

\begin{proof}[Proof of Theorem \ref{thm:mixing}]
    We introduce the shorthand notation
    \[ [B] \coloneqq \{H \in A(d,d-1) \colon H \cap B \neq \emptyset\}, \quad B \subseteq \H^d, \]
    for the set of hyperplanes intersecting a given subset \(B\) of hyperbolic space.
    
    A huge part of the proof given in \cite[Theorem 7.1.4]{HS} for the corresponding fact in Euclidean space carries over directly.
    In particular, one easily verifies that it is enough to show that
    \begin{equation}\label{eq:mixing_proof}
        \P(\mathcal{D}) \to 1 \quad \text{as} \quad \varphi \to \infty,
    \end{equation}
    where the event \(\mathcal D\) is defined as
    \[ \mathcal{D} = \mathcal{D}(\varphi) = \{\xi([B(x,r)] \cap [B(\varphi(x),r)]) = 0\} \]
    with arbitrary but fixed \(x \in \H^d\) and \(r \geq 0\).

    Let \(\varphi_n\) be a divergent sequence in \(\Isom^+(\H^d)\).
    Since the probability of \(\mathcal{D}(\varphi)\) depends only on the distance \(d(x,\varphi(x))\), we can assume without loss of generality, that the points \(\varphi_n(x)\) all lie on the same geodesic ray emanating from \(x\).
    Let \(y^*\) denote the ideal boundary point of this ray.
    For a hyperplane $H$ we write $y^* \in H$ if $y^*$ is an ideal limit point of $H$.
    It holds that
    \begin{align*}
        \limsup_{n\to \infty}[B(x,r)] \cap [B(\varphi_n(x),r)]
            &\subseteq [B(x,r),y^*]\\
            &\coloneqq\{H \in A(d,d-1) \colon H \cap B(x,r) \neq \emptyset, y^* \in H\}.
    \end{align*}
    Indeed, suppose that \(H\) is a hyperplane that does not contain \(y^*\).
    Then there exists a hyperplane \(H'\) separating \(H\) and \(y^*\) (cf.~Figure \ref{fig:sketch_ergodic_proof}).
    Since all but finitely many of the balls \(B(\varphi_n(x),r)\) lie on the side of \(H'\) that contains \(y^*\), they are not intersected by \(H\).
    This give the inclusion (actually, one can show that the inclusion is an equality, but this is not needed for our purposes).
    
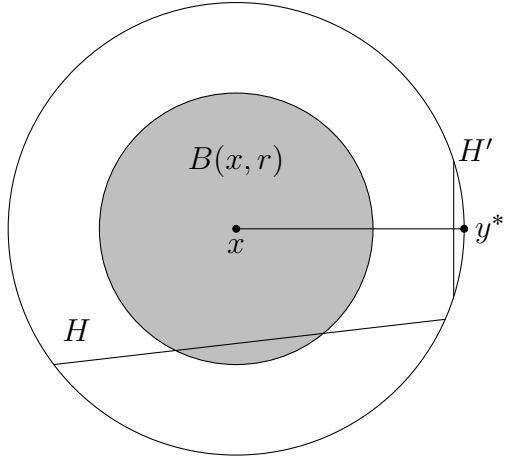
\begin{figure}[ht]
\centering
\begin{tikzpicture}[scale=3.0]
    \filldraw[color=black, fill=white] (0,0) circle (1);
          \filldraw[color=black, fill=lightgray] (0,0) circle (0.6);
          \draw [black] (0.95393,0.3) -- (0.95393,-0.3);
          \draw [black] (-0.8,-.6) -- (0.91651,-0.4);
          \filldraw[fill=black] (0,0) node[below, scale=1] {$x$} circle (0.015);
          \draw (0,0) -- (1.0,0.0);
          \filldraw[fill=black] (1.0,0.0) node[right, scale=1] {$y^*$} circle (0.015);
          \node (A) at (-0.7,-0.45) [scale=1] {$H$};
          \node (A) at (1.05,0.35) [scale=1] {$H'$};
          \node (A) at (0.0,0.3) [scale=1] {$B(x,r)$};
\end{tikzpicture}
\caption{If \(y^* \notin H\), then there exists a hyperplane \(H'\) separating \(y^*\) and \(H\) (illustrated in the Klein model).
\label{fig:sketch_ergodic_proof}}
\end{figure}
    
    From basic properties of the Poisson process it follows that
    \begin{align*}
        &\P(\mathcal{D}(\varphi_n))\\
            &\quad= \exp\left(-\gamma \cdot \mu_{d-1}([B(x,r)] \cap [B(\varphi_n(x),r)])\right)\\
            &\quad= \exp\left(-2\gamma \cdot \int_{S_x^{d-1}} \int_0^r \cosh^{d-1}(\tau)
            \one\{H(u,\tau) \in [B(x,r)] \cap [B(\varphi_n(x),r)]\} \dint\tau \sigma_{d-1}(\dint u) \right),
    \end{align*}
    where we used that \(H(u,\tau) \in [B(x,r)]\) implies \(\tau \leq r\).
    Applying Fatou's lemma now yields
    \begin{multline*}
        \limsup_{n\to\infty}\int_0^r \cosh^{d-1}(\tau)
            \one\{H(u,\tau) \in [B(x,r)] \cap [B(\varphi_n(x),r)]\} \dint\tau \\ \leq
            \int_0^r \cosh^{d-1}(\tau)
            \one\{H(u,\tau) \in [B(x,r),y^*]\} \dint\tau.
    \end{multline*}
    For fixed \(u \in S_x^{d-1}\) and distinct \(\tau,\tau'\geq 0\), the hyperplanes \(H(u,\tau)\) and \(H(u,\tau')\) have no common ideal boundary points.
    Hence, for a given \(u\in S_x^{d-1}\), there exists at most one \(\tau \geq 0\) so that \(H(u,\tau)\) has \(y^*\) as an ideal boundary point.
    It follows that the integral on the right-hand side vanishes, which in turn implies \eqref{eq:mixing_proof}.
\end{proof}

\end{document}